\documentclass[amssymb,amsfonts,refcheck,12pt,verbatim,righttag]{amsart}
\usepackage{color}
\usepackage{amssymb}
\def\R {\Bbb R}
 \numberwithin{equation}{section} 
\newcommand{\beq}{\begin{equation}}
\newcommand{\eeq}{\end{equation}}

\newcommand{\ben}{\begin{eqnarray}}
\newcommand{\een}{\end{eqnarray}}

\newcommand{\bet}{\begin{eqnarray*}}
\newcommand{\eet}{\end{eqnarray*}}

\newtheorem{thm}{Theorem}[section]
\newtheorem{lem}[thm]{Lemma}

\newtheorem{prop}[thm]{Proposition}

\newtheorem{de}[thm]{Definition}
\newtheorem{rem}[thm]{Remark}

\newtheorem{conj}[thm]{Conjecture}

\def\Z{{\Bbb Z}}
\def\C {\pmb{C}}
\def\R {\Bbb R}
\def\N {\Bbb N}

\def\F {{\mathcal F}}
\def\A{{\mathcal A}}

\def\D {{\mathcal D}}

\def\Q{{\Bbb Q}}

\def\sm{[\!]}

\theoremstyle{plain}

\begin{document}
\baselineskip 16pt

\title
{Affine embeddings and intersections of Cantor sets}

\author{De-Jun FENG$\mbox{}^1$}


\author{Wen Huang$\mbox{}^2$}

\author{Hui Rao$\mbox{}^3$}

\date{}

\maketitle
{\small
\begin{center}
$\mbox{}^1$Department of Mathematics, The Chinese University of Hong Kong,\\
Shatin,  Hong Kong.
Email: \email{djfeng@math.cuhk.edu.hk}
\end{center}
\begin{center}
$\mbox{}^2$Department of Mathematics, University of Science and Technology of China\\
 Hefei, Anhui, 230026, P. R.  China.
Email: \email{wenh@mail.ustc.edu.cn}
\end{center}
\begin{center}
$\mbox{}^3$Department of Mathematics, Central China Normal University\\ Wuhan, Hubei,  430070,
P. R. China. Email: \email{hrao@mail.ccnu.edu.cn}
\end{center}
}
\begin{abstract}
Let $E, F\subset \R^d$ be two self-similar sets. Under mild conditions, we show that $F$ can be $C^1$ embedded into $E$ if and only if it can be affinely embedded into $E$; furthermore if $F$ can not be affinely embedded into $E$, then the Hausdorff dimension of the intersection $E\cap f(F)$ is strictly less than that of $F$ for  any $C^1$ diffeomorphism $f$ on $\R^d$. Under certain circumstances, we prove the logarithmic commensurability between the contraction ratios of $E$ and $F$ if $F$ can be affinely embedded into $E$.  As an application, we show that $\dim_HE\cap f(F)<\min\{\dim_HE, \dim_HF\}$ when $E$ is any  Cantor-$p$ set and $F$ any Cantor-$q$ set, where $p,q\geq 2$ are two integers with $\log p/\log q\not \in \Q$. This  is related to   a conjecture of Furtenberg about the intersections of Cantor sets.\\

\noindent R\'{E}SUM\'{E}.
Soit $E$ et $F$ deux ensembles auto-similaires dans $\R^d$. Sous des hypoth\`eses raisonnables, nous montrons qu'il existe un plongement $C^1$ de $F$  dans $E$ si et seulement s'il existe un tel plongement affine; de plus, s'il n'existe pas de plongement affine de $F$ dans $E$, alors pour tout diff\'eomorphisme $C^1$ de $\R^d$ la dimension de Hausdorff de l'intersection $E\cap f(F)$ est strictement inf\'erieure \`a celle de $F$. Dans certains cas, nous montrons que les logarithmes des facteurs de contraction de $E$ et $F$ sont commensurables lorsqu'il existe un plongement affine de $F$ dans $E$. En application, nous montrons que $\dim_H E\cap f(F)<\min\{\dim_H E,\dim_H F\}$ quand $E$ est un $p$-ensemble de Cantor et $F$ est un $q$-ensemble de Cantor, o\`u $p$ et $q$ sont des nombres entiers $\ge 2$ tels que $\log p/\log q\not\in \mathbb Q$. Ceci est reli\'e \`a une conjecture de Furstenberg sur les intersections d'ensembles de Cantor.

\end{abstract}

\setcounter{section}{0}
\section{Introduction}
\setcounter{equation}{0}

Let $A,B$ be two subsets of $\R^d$. We say that $A$ can be {\it affinely embedded} into $B$ if $f(A)\subseteq B$ for some affine map $f: \R^d\to \R^d$ of the form $f(x)=Mx+a$, where $M$ is an invertible $d\times d$ matrices and $a\in \R^d$. Similarly, we say that $A$ can be {\it $C^1$-embedded} into $B$ if $f(A)\subseteq B$ for some $C^1$-diffeomorphism $f$ on $\R^d$.

The objective of this paper is to study the relation between $C^1$-embeddings and affine embeddings for self-similar sets; and to study the necessary conditions under which  one self-similar set  can be affinely embedded or $C^1$-embedded into another self-similar set.
These questions are motivated from several projects in related areas, including the classification of self-similar subsets of  Cantor sets \cite{FRW12}, the characterization of  Lipschitz equivalence  and Lipschitz embedding of Cantor sets \cite{FaMa92, DWXX02},   as well as the study of intersections of Cantor sets \cite{Fur69, EKM10} and the geometric rigidity of $\times m$  invariant measures \cite{Hoc12'}.

Before stating our results, we recall some terminologies about self-similar sets. Let $\Phi=\{\phi_i\}_{i=1}^\ell$ be  a finite family of contractive mappings on $\R^d$.
 Following Barnsley \cite{Bar88}, we say that $\Phi$ is an {\it  iterated function system} (IFS) on $\R^d$.
 Hutchinson \cite{Hut81} showed that there is a unique non-empty compact set $K\subset \R^d$, called the {\it attractor} of $\Phi$,   such that
$$K=\bigcup_{i=1}^\ell \phi_i(K).$$
Correspondingly, $\Phi$ is called a {\it generating IFS} of $K$. One notices that $K$ is a singleton if and only if the mappings $\phi_i$, $1\leq i\leq \ell$, have the same fixed point.
We say that $\Phi$ satisfies the {\it open set condition} (OSC) if there exists a non-empty bounded open set $V\subset \R^d$ such that $\phi_i(V)$, $1\leq i\leq \ell$, are pairwise disjoint subsets of $V$. Similarly, we say that $\Phi$ satisfies the {\it strong separation condition} (SC) if $\phi_i(K)$ are pairwise disjoint subsets of $K$. The SC always implies the OSC.

A mapping $\phi:\R^d\to \R^d$ is called a {\it similitude} if $\phi$ is of the form $\phi(x)=\alpha R(x)+a$ for $x\in \R^d$, where $\alpha>0$, $R$ is an orthogonal transformation and $a\in \R^d$.  When all maps in an IFS $\Phi$ are similitudes, the attractor  $K$ of $\Phi$ is called a {\it self-similar} set; in this case, the {\it self-similar dimension} of $K$ is defined as the unique positive number $s$ so that $\sum_{i=1}^\ell \rho_i^s=1$, where $\rho_i$ denotes the contraction ratio of $\phi_i$. It is well known \cite{Hut81} that $\dim_HK=s$ if $\Phi$ consists of similitudes and satisfies the OSC, here $\dim_H$ denotes the Hausdorff dimension; the condition of OSC can be further replaced by some significantly weaker separation condition in the case $d=1$ and $s\leq 1$ \cite{Hoc12}.

In the remaining part of this section, we assume that  $\Phi=\{\phi_i\}_{i=1}^\ell$ and $\Psi=\{\psi_j\}_{j=1}^m$ are two families of contractive similitudes of $\R^d$ of the form
\begin{equation}
\label{e-1.1}
\phi_i(x)=\alpha_iR_i(x)+a_i,\quad \psi_j(x)=\beta_j O_j(x)+b_j, \quad i=1,\ldots, \ell,\; j=1,\ldots, m,
\end{equation}
where $0<\alpha_i, \beta_j<1$, $a_i, b_j\in \R^d$ and $R_i, O_j$ are orthogonal transformations on $\R^d$.  Let $E$, $F$ be the attractors of  $\Phi$ and $\Psi$, respectively. To avoid  triviality, we always assume that $E, F$ are not singletons in this paper.

For any real invertible $d\times d$  matrix $M$, let $\kappa(M)$ denote the {\it condition number} of $M$, that is,
\begin{equation}
\label{e-1.2}
\kappa(M)=\max\left\{\frac{|Mu|}{|Mv|}:\; u, v\in \R^d \mbox{ with }|u|=|v|=1\right\}.
\end{equation}

The first result of this paper is the following.

\begin{thm}
\label{thm-1.1}  Assume that $\Phi$ satisfies the OSC, and the Hausdorff dimension of $F$ equals its self-similar dimension. Then $F$ can be $C^1$-embedded into $E$ if and only if $F$ can be affinely embedded into $E$. Furthermore if $F$ can not  be affinely embedded into $E$, then  $$\dim_H(E\cap f(F))<\dim_HF$$ for any $C^1$ diffeomorphism $f$ on $\R^d$; moreover, for any $L>0$,
 $$
 \sup_{f\in {\rm Diff^1_L(\R^d)}} \dim_H(E\cap f(F))<\dim_HF,
 $$
 where ${\rm Diff^1_L(\R^d)}$ denotes the collection of all $C^1$ diffeomorphisms $f$ on $\R^d$ so that $\kappa(D_x(f^{-1}))\leq L$ for any $x\in E$, in which  $D_x(f^{-1})$ denotes the differential of $f^{-1}$ at $x$.
\end{thm}

The proof of Theorem \ref{thm-1.1} is mainly based on the similarity of self-similar sets. In what follows, we discuss when $F$ can be affinely embedded into $E$.
It is natural to  expect that if $E, F$ are  totally disconnected and  $F$ can be affinely embedded into $E$, then the ratios $\alpha_i, \beta_j$ should satisfy  some arithmetic conditions.  We formulate the following  conjecture from  this view point.
\begin{conj}
\label{conj1}
Suppose that  $E, F$ are  totally disconnected and   $F$ can be affinely embedded into $E$.  Then for each $1\leq j\leq m$,  there exist non-negative rational numbers $t_{i,j}$ such that
$
\beta_j=\prod_{i=1}^\ell\alpha_i^{t_{i,j}}.
$
In particular, if $\alpha_i=\alpha$ for  $1\leq i\leq \ell$, then $\log \beta_j/\log \alpha\in \Q$ for $1\leq j\leq m$.
\end{conj}

We remark that  the above arithmetic conditions on   $\alpha_i, \beta_j$  do fulfill if  $E$ and $F$ are Lipschitz equivalent and dust-like (i.e.,  $\Phi$, $\Psi$ satisfy the SC) \cite{FaMa92}. Nevertheless, no arithmetic conditions are needed for the Lipschitz embeddings. It was shown in \cite{DWXX02} that
if $E, F$ are dust-like with $\dim_HF<\dim_HE$, then  $F$ can be Lipschitz embedded into $E$.

In the following we give some partial answers to Conjecture \ref{conj1}.
\begin{thm}Assume that $\Phi$ satisfies the SC,   $\alpha_i=\alpha$ for $1\leq i\leq \ell$, and $\dim_HE<1/2$. If   $F$ can be affinely embedded into $E$, then   $\log \beta_j/\log \alpha\in \Q$ for $1\leq j\leq m$.
\label{thm-1.3}
\end{thm}

The main idea in the proof of Theorem \ref{thm-1.3} is to show that if  $F$ can be affinely embedded into $E$ but  $\log \beta_j/\log \alpha\not\in \Q$ for some $j$, then the set $\{|x-y|:\; x,y\in E\}$ contains a non-degenerate interval, which contradicts the assumption that $\dim_HE<1/2$. The argument involves the theory of compact Lie groups.

We can further sharpen the above result when both $E$ and $F$ are central Cantor sets in $\R$. For $0<\rho<1/2$, let $\C_\rho$ denote the attractor of the IFS
$\{\rho x,\; \rho x+ 1-\rho\}$. It is easy to check that
$$\C_\rho=\left\{ x=\sum_{i=0}^{+\infty}\epsilon_i(1-\rho)\rho^i:\epsilon_i\in \{0,1\} \text{ for all
}i\ge 0\right\}.$$
Recall that a {\it Pisot number} is an algebraic integer $>1$  whose
algebraic conjugates
are all inside the unit disk. For instance,  $\sqrt{2}+1$ is a Pisot number (it has a unique algebraic conjugate $\sqrt{2}-1$), so are the positive roots of $x^n-x^{n-1}-\cdots-x-1$ for $n\geq 2$  and the positive roots of $x^{2n}(x-2)-1$ for $n\geq 1$. Of course, all integers greater than 1 are Pisot numbers. The readers are referred to \cite{Sal63} for further properties of Pisot numbers.

\begin{thm}
\label{thm-1.4}
Let $0<\beta<\alpha<1/2$. Then the following statements hold.
\begin{itemize}
\item[(i)] If $\alpha<1/4$, then $\C_\beta$ can be affinely embedded into $\C_\alpha$ if and only if $\log \beta/\log \alpha\in \N$.
\item[(ii)] If $1/4\leq \alpha< \sqrt{2}-1$ and $\C_\beta$ can be affinely embedded into $\C_\alpha$, then $\log \beta/\log \alpha\in \Q$. However, it is possible that $\log \beta/\log \alpha\not\in \N$.
    \item[(iii)] If $1/\alpha$ is a Pisot number  and $\C_\beta$ can be affinely embedded into $\C_\alpha$, then $\log \beta/\log \alpha\in \Q$; furthermore $1/\beta$ is a Pisot number.
\end{itemize}
\end{thm}

The following result is an extension of  Theorem \ref{thm-1.4}(iii).
\begin{thm}
\label{thm-1.5}
Assume that $\theta=1/\alpha$ is a Pisot number $>2$. Furthermore  in \eqref{e-1.1}, assume that $d=1$, $2\leq \ell<\theta$,  $\alpha_iR_i(x)=\alpha x$ and  $a_i\in \Z[\theta]$ for $1\leq i\leq \ell$, here $\Z[\theta]$ denotes the ring of $\theta$ over $\Z$.  If  $F$ can be affinely embedded into $E$, then  $\log \beta_j/\log \alpha\in \Q$ and $1/\beta_j$ is a Pisot number for $j=1,\ldots, m$.
\end{thm}

Our results are related to one of the conjectures of Furstenberg about the intersections of Cantor sets \cite{Fur69}. Let $p\in \N$ with $p\geq 2$. Following Furstenberg,
 we call  $A$ a {\it Cantor $p$-set} if $A$ is the attractor of an IFS $\{x/p+a_i\}_{i=1}^\ell$ on $\R$, where $\{a_i: 1\leq i\leq \ell\}$ is a proper subset of $\{0, 1, \ldots,p-1\}$ containing at least two digits. Furstenberg conjectured that if $p, q$ are not powers of the same integer (i.e.,  $\frac{\log q}{\log p}\not \in \mathbb{Q})$, then
 $$
 \dim_H(A\cap f(B))\leq \max\{0, \dim_HA+\dim_HB-1\},
 $$
where $A$ is an arbitrary Cantor $p$-set and $B$ a Cantor $q$-set, $f$ is any affine map on $\R$.  So far this conjecture is still open in its full generality. As a corollary of Theorems \ref{thm-1.1} and \ref{thm-1.5}, we have  the following  related result, although  it is still far from Furstenberg's  conjecture.
\begin{thm}
Suppose that  $p, q\geq 2 $ are not powers of the same integer. Then for any  Cantor $p$-set $A$ and Cantor $q$-set $B$, we have
$$\sup_f\dim_H(A\cap f(B))<\min\{\dim_HA, \dim_HB\},$$
  where the supremum is taken over the set of  $C^1$ diffeomorphisms  on $\R$.
\end{thm}

For the convenience of the readers, we illustrate the rough ideas in the proofs of Theorem \ref{thm-1.4} (ii) and (iii). Assume that $\C_\beta$ can be affinely embedded into $\C_\alpha$ but $\log \beta/\log \alpha\not\in \Q$. Then by using the self-similarity structure of $\C_\alpha$, $\C_\beta$ and the irrationality of $\log \beta/\log \alpha$, we can  show that for any $\lambda\in  (0, \frac{1-2\alpha}{\alpha}]$, there exists $c=c(\lambda)\in \C_\alpha$ such that
\begin{equation}
\label{e-1.3}
\lambda \C_\beta+c\subset \C_\alpha.
\end{equation}
Furthermore, we can show that for any $c$ so that \eqref{e-1.3} holds,  the symbolic coding of $c$ can not be periodic.
To derive a contradiction, we first consider the case that $\alpha<\sqrt{2}-1$. By considering the beta expansions in base $\alpha$, we can show that there exists $u\in(0, \frac{1-2\alpha}{\alpha}]$, so that there is a unique $c\in \C_\alpha$ so that $u+c\in \C_\alpha$. Furthermore, the symbolic coding of such $c$ is periodic. However by \eqref{e-1.3}, we should have $u\C_\beta+c\subset \C_\alpha$, and thus the symbolic coding of such $c$ can not be periodic. This leads to a contradiction.
Next we consider the case that  $1/\alpha$ is a Pisot number $>2$. Our argument involves classic harmonic analysis. By a well-known result of Salem and Zygmund,  $\C_\alpha$ is a set of uniqueness (cf. Definition \ref{de-5.1}) and hence $\C_\alpha$ does not support any Borel probability measure whose Fourier transform vanishes at infinity (cf. Theorem \ref{thm-5.2}). However, we can use \eqref{e-1.3} to construct a measure on $\C_\alpha$ whose Fourier transform vanishes at infinity, leading to a contradiction.

 We remark that  Theorem \ref{thm-1.4} extends a previous result in \cite{FRW12}: if $\C_\beta$ can be affinely embedded into $\C_{1/3}$, then $1/\beta$ should be an integer power of $3$. The idea used in \cite{FRW12} can be extended to  prove that $\log \beta/\log\alpha \in \Q$ if $\alpha\leq 1/3$ and $\C_\beta$ can be affinely embedded into $\C_\alpha$; however it can not deal with the case $\alpha>1/3$. After we got an initial draft of this paper, Pablo Shmerkin informed us an alternative dynamical approach in proving  the first part of Theorem \ref{thm-1.4}(iii), which is based on  the  general development in \cite{HoSh12} about equidistributions of fractal measures.

 The paper in organized as follows. In Section \ref{S-2}, we prove Theorem \ref{thm-1.1}. In  Section \ref{S-3}, we prove  Theorem \ref{thm-1.3} and Theorem \ref{thm-1.4}(i). In  Section \ref{S-4}, we prove Theorem \ref{thm-1.4}(ii). In Section \ref{S-5}, we prove Theorem \ref{thm-1.4}(iii) and Theorem \ref{thm-1.5}.

 \section{Relation between affine embeddings and $C^1$-embeddings}
 \label{S-2}

In this section we study the relation between $C^1$-embeddings and affine embeddings of self-similar sets, and prove Theorem \ref{thm-1.1}.

Let  $\Phi=\{\phi_i=\alpha_iR_i+a_i\}_{i=1}^\ell$ and $\Psi=\{\psi_j=\beta_jO_j+b_j\}_{j=1}^m$  be two IFSs of the form \eqref{e-1.1}, and $E, F$  the corresponding attractors. Assume that $\Phi$ satisfies the OSC, and the Hausdorff dimension of $F$ equals its self-similar dimension, i.e., $\dim_HF=s$ with $\sum_{i=1}^m\beta_i^s=1$. With loss of generality, assume that $$\alpha_1=\min\{\alpha_i: 1\leq i\leq \ell\},\qquad  \beta_1=\min\{\beta_j: 1\leq j\leq m\}.$$

Write $\phi_I=\phi_{i_1}\circ\cdots\circ \phi_{i_n}$ and $\alpha_I= \alpha_{i_1}\cdots \alpha_{i_n}$ for $I= i_1\ldots i_n\in \{1,\ldots, \ell\}^n$. Similarly, we use the abbreviations $\psi_J$ and $\beta_J$ for $J\in \{1,\ldots, m\}^n$.

For any $n\in \N$, let $s_n$ be the unique positive number satisfying
\begin{equation}
\label{e-2.0}
\sum_{J\in \{1,\ldots,m\}^n: J\neq 1^n} \beta_J^{s_n}=1.
\end{equation}
 That is, $(\sum_{i=1}^m\beta_i^{s_n})^n-\beta_1^{ns_n}=1$. Then
\begin{equation}
\label{e-2.1}
\sum_{J\in \Gamma}\beta_{J}^{s_n}\leq 1 \quad \mbox{for any proper subset $\Gamma$ of $\{1,\ldots, m\}^n$}.
\end{equation}
Clearly, $\lim_{n\to \infty}s_n=s$.

For any $0<r<\alpha_1$, denote
$$
{\mathcal A}_r:=\{I=i_1\ldots i_n\in \{1,\ldots, \ell\}^n:\; n\in \N,\; \alpha_{i_1}\ldots \alpha_{i_n}\leq r<\alpha_{i_1}\ldots \alpha_{i_{n-1}}\}.
$$

\begin{lem}
\label{lem-2.1}
There exists $N_0\in \N$ such that for any $0<r<\alpha_1$ and  $I\in \A_r$,
$$\#\{J\in \A_r:\; {\rm dist}(\phi_I(E), \phi_{J}(E))\leq r\}\leq N_0,$$
where $\# A$ denotes the cardinality of $A$.
\end{lem}
\begin{proof}
Since $\Phi$ satisfies the OSC, there exists a non-empty bounded open set $V\subset \R^d$ such that $\phi_i(V)$ ($1\leq i\leq \ell$) are disjoint subsets of $V$.
Clearly, $\bigcup_{i=1}^\ell\phi_i(\overline{V})\subseteq \overline{V}$;  hence $E\subseteq \overline{V}$.  It is not hard to check that for any $0<r<\alpha_1$,\; $\phi_I(V)$ ($I\in \A_r$) are disjoint subsets of $V$.
Since $V$ is a bounded open set, there are two closed balls $B_1, B_2$ such that $B_1\subset V\subset B_2$. Let $R_1, R_2$ denote the radii of $B_1$ and $B_2$ respectively.

 Now fix $0<r<\alpha_1$ and  $I\in \A_r$.  Let $J_1, \ldots, J_k$ be elements in $\A_r$ so that ${\rm dist}(\phi_I(E), \phi_{J_t}(E))\leq r$ for $t=1,\ldots, k$. Then ${\rm dist}(\phi_I(\overline{V}), \phi_{J_t}(\overline{V}))\leq r$. Hence
 $\phi_{J_t}(V)$ ($1\leq t\leq k$) are contained in a ball of radius $4rR_2+r$, and each of them contains a ball  of radius $\geq r \alpha_1R_1$.  A volume argument shows that
 $$
 k\leq \frac{(1+4R_2)^d}{\alpha_1^d R_1^d}.
 $$
 This finishes the proof of the lemma  by taking  $N_0$ to be an integer greater than the right-hand side of the above inequality.
\end{proof}

For any $d\times d$ real  matrix $M$, we use $\|M\|$ to denote the usual norm of $M$, and  $\sm M\sm$ the smallest singular value of $M$, i.e.,
\begin{equation}
\label{e-M1}
\begin{split}
\|M\| &=\max\{|Mv|: \;
v\in \R^d, |v|=1\} \quad \mbox{and}\\
 \sm M\sm &=\min\{|Mv|: \;
v\in \R^d, |v|=1\}.
\end{split}
\end{equation}

 Recall that ${\rm Diff^1_L(\R^d)}$ denotes the collection of all $C^1$ diffeomorphisms $f$ on $\R^d$ so that the condition number of $D_x(f^{-1})$ does not exceed $L$ (i.e. $\| D_x(f^{-1})\| \leq L\sm D_x(f^{-1}) \sm$)  for any $x\in E$. The following proposition plays a key role in our proof of Theorem \ref{thm-1.1}.
 \begin{prop}
 \label{pro-2.1}
 Let $n\in \N$ and  $f\in {\rm Diff}^1_L(\R^d)$ for some $L>1$. Assume that
 $$
 \dim_H(E\cap f(F))>s_n,$$
 where $s_n$ is given as in \eqref{e-2.0}. Let $N_0$ be the integer given in Lemma \ref{lem-2.1}.
 Then there exist affine mappings $g_1,\ldots, g_k$ with $k\leq N_0$ such that the linear parts of $g_i$ have condition number not exceeding $L$ and
 $$
 \psi_{J}(F)\cap \left(\bigcup_{i=1}^kg_i(E)\right)\neq \emptyset, \qquad \forall J\in \{1,\ldots, m\}^n.
 $$
 \end{prop}

 \begin{proof} Denote $h=f^{-1}$. Then $h$ is a $C^1$ diffeomorphism on $\R^d$. Hence there exists $\delta>0$ so that for any $x,y\in E$ with $|h(x)-h(y)|<\delta$, we have $|(D_xh)(x-y)|/2\leq |h(x)-h(y)|$
 and hence
 \begin{equation}
 \label{e-2.3}
 \sm D_xh\sm \cdot |x-y|/2\leq  |h(x)-h(y)|.
 \end{equation}

 Since $h$ is a $C^1$ diffeomorphism on $\R^d$,   $\dim_H(F\cap h(E))=\dim_H(f(F)\cap E)>s_n$. We claim that for any $j\in \N$, there exists a word $W_j$ on $\{1,\ldots,m\}$ with length $|W_j|\geq j$ such that
\begin{equation}
\label{e-2.2}
 \psi_{W_jJ}(F)\cap h(E)\neq \emptyset, \qquad \forall\; J\in \{1,\ldots, m\}^n.
 \end{equation}
 Assume on the contrary that the above claim is false. Then there exists $p_0\in \N$ such that for any word $W$ on $\{1,\ldots,m\}$ with length $|W|\geq p_0$, there exists at least one
 $J\in \{1,\ldots, m\}^n$ such that $\psi_{WJ}(F)\cap h(E)= \emptyset$. For $q\in \N$, denote
 $$
 t(q):=\sum_{UJ_1\cdots J_q\in \Gamma_q} ({\rm diam}(\psi_{UJ_1\cdots J_q}(F)))^{s_n}=\sum_{UJ_1\cdots J_q\in \Gamma_q} ({\rm diam}(F))^{s_n} (\beta_{UJ_1\cdots J_q})^{s_n},
 $$
 where  $\Gamma_q$ denotes the set of  words $UJ_1\ldots J_q$ on $\{1,\ldots, m\}$ so that $|U|=p_0$, $|J_i|=n$  for $1\leq i\leq q$,  and $\psi_{UJ_1\ldots J_q}(F)\cap h(E)\neq \emptyset$. Notice that for any word $UJ_1\ldots J_{q-1}\in \Gamma_{q-1}$, there exists at least one $J$ with $|J|=n$ so that $UJ_1\ldots J_{q-1}J\not\in  \Gamma_{q}$. Hence by \eqref{e-2.1}, we have $t(q)\leq t(q-1)\leq \ldots\leq t(1)$. Since for each $q\in \N$, $\{\psi_{I}(F):\; I\in \Gamma_q\}$ is a cover of $F\cap h(E)$, we have $\dim_H(F\cap h(E))\leq s_n$, leading to a contradiction. This proves our claim \eqref{e-2.2}.

According to \eqref{e-2.2}, we have
\begin{equation}
\label{e-2.4}
 \psi_{J}(F)\cap \psi_{W_j}^{-1}\circ h(E)\neq \emptyset\qquad \mbox{ for all } j\in \mathbb{N},\; J\in \{1,\ldots, m\}^n.
\end{equation}
Denote $\rho=\min_{x\in E} \sm D_xh\sm$. Pick  $p_1\in \mathbb{N}$ so that
$$2{\rm diam}(F) (\max_{1\leq i\leq m}\beta_i)^{p_1}< \min\left\{\frac{2\delta}{\rho},\alpha_1\right\} \cdot \rho.$$
For any $j\geq p_1$, denote $r_j=2{\rm diam}(F) \beta_{W_j}\rho^{-1}$. Then $0<r_j<\min\{\frac{2\delta}{\rho},\alpha_1\}$. By \eqref{e-2.3}, we derive that if $x, y\in E$ with $|x-y|>r_j$, then $|h(x)-h(y)|>\rho r_j/2$. To see this, notice that $\rho r_j/2<\delta$; if $|h(x)-h(y)|<\rho r_j/2$ then by \eqref{e-2.3}, $$|x-y|\leq  2|h(x)-h(y)|/\sm D_xh\sm\leq 2(\rho r_j/2)/\rho=r_j,$$ leading to a contradiction. Hence for $x, y\in E$ with $|x-y|>r_j$,
$$|\psi_{W_j}^{-1}\circ h(x)-\psi_{W_j}^{-1}\circ h(y)|=\beta_{W_j}^{-1}|h(x)-h(y)|> \beta_{W_j}^{-1}\rho r_j/2={\rm diam}(F).$$
As a consequence, if $I, J\in \A_{r_j}$ satisfy ${\rm dist}(\phi_I(E), \phi_J(E))>r_j$, then
$${\rm dist} (\psi_{W_j}^{-1}\circ h \circ \phi_I(E),\; \psi_{W_j}^{-1}\circ h \circ \phi_J(E))>{\rm diam}(F);$$
thus at most one of $\psi_{W_j}^{-1}\circ h \circ \phi_I(E)$, $\psi_{W_j}^{-1}\circ h \circ \phi_J(E)$ can intersect $F$.
This combining with Lemma \ref{lem-2.1} yields that
$$\#\{I\in \A_{r_j}:\; \psi_{W_j}^{-1}\circ h\circ  \phi_I(E)\cap F\neq \emptyset\}\leq N_0.$$
Let $I_{j,1}, \ldots, I_{j,{k_j}}$ be all the words in $ \A_{r_j}$ so that $\psi_{W_j}^{-1}\circ h\circ  \phi_{I_{j, t}}(E)\cap F\neq \emptyset$ for $1\leq t\leq k_j$.
Then $k_j\leq N_0$. By \eqref{e-2.4},
\begin{equation}
\label{e-2.5}
 \psi_{J}(F)\cap \left(\bigcup_{t=1}^{k_j}\psi_{W_j}^{-1}\circ h\circ \phi_{I_{j, t}}(E)\right)\neq \emptyset\qquad \mbox{ for all } j\geq p_1,\; J\in \{1,\ldots, m\}^n.
\end{equation}
Notice that $\alpha_1 r_j\leq \alpha_{I_{j, t}}\leq r_j$ for $1\leq t\leq k_j$ and  $r_j=2{\rm diam}(F) \beta_{W_j}\rho^{-1}$. We  have
\begin{equation}
\label{e-2.5'}
2{\rm diam}(F) \rho^{-1}\alpha_1<\beta_{W_j}^{-1}\alpha_{I_{j, t}}\leq 2{\rm diam}(F) \rho^{-1}.
\end{equation}

 For each $j\geq p_1$ and $1\leq t\leq k_j$, we denote $h_{j,t}:=\psi_{W_j}^{-1}\circ h\circ \phi_{I_{j, t}}$. The mappings $h_{j,t}$ can be viewed as a kind of rescalings of $h$.  Fix $x_0\in E$. Define affine mappings $A_{j,t}$ on $\R^d$ by
 $$
 A_{j,t}(x)=h_{j,t}(x_0)+ D_{x_0}h_{j,t}(x-x_0),
 $$
 where $D_{x_0}h_{j,t}$ denotes the differential of $h_{j,t}$ at $x_0$.  By our assumption on $h$, we have
 \begin{equation}
 \label{e-2.5''}
 \|D_{x_0}h_{j,t}\|/\sm D_{x_0}h_{j,t}\sm\leq L.
 \end{equation}
 According to \eqref{e-2.5'}, there exist positive constants $c_1, c_2$ (independent of $j$ and $t$) so that
 \begin{equation}
 \label{e-2.5-1}
  c_1\leq \sm D_{x_0}h_{j,t} \sm \leq \| D_{x_0}h_{j,t}\|\leq c_2.
 \end{equation}

 Since $h$ is a $C^1$-diffeomorphism on $\R^d$ and $E$ is compact, there exists a sequence  $(d_j)$ of positive numbers   with $d_j\downarrow 0$ as $j\to \infty$, such that
 $$
 |h(u)-h(v)-(D_vh)(u-v)|\leq d_j |u-v|
 $$
 for any $u, v\in E\mbox{  with } |u-v|\leq r_j {\rm diam}(E)$. It follows that for any $j\geq p_1$, $1\leq t\leq k_j$ and $x\in E$,
 $$
 |h\circ \phi_{I_{j, t}}(x)-h\circ \phi_{I_{j, t}}(x_0)-(D_{x_0} (h\circ \phi_{I_{j, t}})) (x-x_0)|\leq d_j \alpha_{I_{j, t}} |x-x_0|\leq d_j r_j{\rm diam}(E).
 $$
 Hence we have for each $j\geq p_1$,
 \begin{equation}
 \label{e-2.5-2}
 \begin{split}
 &\sup_{x\in E,\; 1\leq t\leq k_j} \| A_{j,t}(x)-h_{j,t}(x)\|\\
  & \mbox{}\quad=\beta_{W_j}^{-1}\sup_{x\in E,\; 1\leq t\leq k_j}\left|h\circ\phi_{I_{j, t}}(x)-h\circ \phi_{I_{j, t}}(x_0)-(D_{x_0} (h\circ \phi_{I_{j, t}})) (x-x_0)\right|\\
  &\mbox{}\quad  \leq d_j \beta_{W_j}^{-1}r_j{\rm diam}(E)= d_jC,
 \end{split}
 \end{equation}
 where  $C:=2{\rm diam}(E){\rm diam}(F) \rho^{-1}$.

 For $j\geq p_1$ and $1\leq t\leq k_j$, since $h_{j,t}(E)\cap F\neq \emptyset$, by \eqref{e-2.5-1}-\eqref{e-2.5-2}, we see that
 the translation part of  $A_{j,t}$ is uniformly bounded. Combining this fact with
 \eqref{e-2.5-1}, we see that for each sequence of indices $(j_\ell, t_\ell)$ with $1\leq t_\ell\leq k_{j_\ell}$,
 there exists a subsequence $(j_{\ell'}, t_{\ell'})$ so that $A_{j_{\ell'},t_{\ell'}}$ converges to some affine map $g$ on $\R^d$; by \eqref{e-2.5''} and
 \eqref{e-2.5-2},  the linear part of $g$ has condition number $\leq L$, and
 $h_{j_{\ell'},t_{\ell'}}(E)$   converges to  $g(E)$ in Hausdorff distance as $\ell'\to \infty$.

  As a refinement of the above argument, we see that there exists a subsequence $(j_\ell)$ of $\N$ such that $k_{j_\ell}\equiv k$ for some $k\leq N_0$, and moreover for each $1\leq t\leq k $, $A_{j_\ell, t}\to g_t$ for some affine map $g_t$ as $\ell\to \infty$.
  In particular,
  $$\bigcup_{t=1}^{k_{\ell}}\psi_{W_{j_\ell}}^{-1}\circ h\circ \phi_{I_{j_\ell, t}}(E)=\bigcup_{t=1}^k h_{j_\ell, t}(E)\to \bigcup_{t=1}^k g_t(E)$$
 in Hausdorff distance as $\ell \to \infty$.  Now the proposition follows from \eqref{e-2.5}.
  \end{proof}

\begin{proof}[Proof of Theorem \ref{thm-1.1}] It suffices to show that for any $L>1$, if $$
 \sup_{f\in {\rm Diff^1_L(\R^d)}} \dim_H(E\cap f(F))=\dim_HF,
 $$
then $F$ can be affinely embedded into $E$. Indeed assume that the above identity holds. Then by Proposition \ref{pro-2.1}, for any $n\in \N$, there exists a family of affine mappings $\{g_{i}^{(n)}\}_{i=1}^{k_n}$ so that  $k_n\leq N_0$,  the linear parts of $g^{(n)}_i$ have condition number $\leq L$, $g_{i}^{(n)}(E)\cap F\neq \emptyset$, and
$$
\psi_{J}(F)\cap \left(\bigcup_{i=1}^{k_n}g_i^{(n)}(E)\right)\neq \emptyset \qquad \forall J\in \{1,\ldots, m\}^n.
$$
A compactness argument shows that there exist affine maps $g_1,\ldots, g_k$ with $k\leq N_0$ such that
$F\subseteq \bigcup_{i=1}^kg_i(E)$. A version of Baire category theorem states that there exist an open set $V\subset \R^d$ and $i\in \{1,\ldots, k\}$ such that
$\emptyset \neq F\cap V\subset g_i(E)$.  However, $F\cap V\supset \psi_J(F)$ for some word $J$ on $\{1,\ldots, m\}$; hence $F\subset \psi_J^{-1}\circ g_i(E)$, i.e.,  $F$ can  be affinely embedded into $E$.
\end{proof}

\section{Affine embeddings and the logarithmic commensurability}
\label{S-3}
In this section, we will prove  Theorem \ref{thm-1.3} and Theorem \ref{thm-1.4}(i).

\begin{proof}[Proof of Theorem \ref{thm-1.3}]
Since  $F$ can be affinely embedded into $E$,  there is an affine map $g(x)=Mx+b$ with $\det(M)\neq 0$ such that $g(F)\subseteq E$.

We first consider the case that  $\beta_j=\beta$ for $j=1,\ldots, m$.  Assume on the contrary that  $\frac{\log \beta}{\log\alpha}\not\in \Q$. We show below that   $\dim_H(E-E)\ge 1$, which implies that $$2\dim_HE=\dim_HE\times E\geq \dim_H(E-E)\geq 1,$$
i.e., $\dim_HE\geq 1/2$, leading to a contradiction.

 Let $\delta:=\min_{i\neq j} d(\phi_i(E),\phi_j(E))$ and $\Gamma:=\max\{\text{diam}(MO(F)):O\in {\mathcal O}(d)\}$, where ${\mathcal O}(d)$ denotes the collection of orthogonal transformations on $\R^d$. Then
$0<\delta,\Gamma<+\infty$. Fix $p,N\in \mathbb{N}$ such that $\alpha^p<\frac{\delta}{\Gamma}$ and $\frac{\log \beta}{\log \alpha}N\ge p$.

Now for $n\in \mathbb{N}$ with $n\ge N$, $g(\psi_{1^n}(F))\subset g(F)\subseteq E$. Notice that $\psi_{1^n}(F)$ is of the form $\beta^nO_1^n(F)+e_n$ for some $e_n\in\R^d$.  We have $\beta^n M O_1^n(F)+b'\subseteq E$ with $b':=Me_n+b$.
Let $\ell_n$ be the integer part of $\frac{\log \beta}{\log \alpha}n$. Then $${\rm diam}(\beta^n M O_1^n(F)+b')\le \beta^n\Gamma<\delta \alpha^{\ell_n-p}.$$
By the definition of $\delta$, we see that $\beta^n M O_1^n(F)+b'$ intersects $\psi_{I}(E)$ for only one $I\in \{1,\ldots, \ell\}^{\ell_n-p}$, and thus
$$\beta^n M O_1^n(F)+b'\subseteq \psi_{I}(E).$$
Hence there exists $P_n\in {\mathcal O}(d)$ and $r_n\in \R^d$ such that $\beta^n M O_1^n(F)\subseteq \alpha^{\ell_n-p}P_n(E)+r_n$. This implies
$$\alpha^{p+\{\gamma n\}}P_n^{\tau}MO_1^n(F)-\frac{r_n}{\alpha^{\ell_n-p}}\subseteq E,$$
where $\gamma:=\frac{\log \beta}{\log \alpha}$, and $\{x\}$ denotes the fractional part of $x$. Fix a nonzero vector $v\in F-F$.
Then
$$E-E\supseteq \alpha^{p+\{\gamma n\}}P_n^{\tau}MO_1^n(F-F)\supseteq \alpha^{s+\{\gamma n\}}P_n^{\tau}MO_1^nv$$
for $n\ge N$. Denote $U:=\{|x_1-x_2|^2:\;x_1,x_2\in E\}$. Then
\begin{equation}\label{eq-11}
U\supseteq \{\alpha^{2(p+\{\gamma n\})}|MO_1^nv|^2:n\ge N\}.
\end{equation}

Now we consider the closure $W$ of $\{(e^{2\pi i n\gamma},O_1^n):n\ge N\}$ in the compact Lie group $S^1\times {\mathcal O}(d)$.
It is clear that $W$ is a closed subgroup of $S^1\times {\mathcal O}(d)$. Hence by Cartan Theorem (cf.  \cite[Theorem 3.3.1]{Pri77}),   $W$ is also a Lie group. Let $W_0$ be the connected component of $W$ containing the unit element $(1,I)$.
Then $W_0$ is a connected compact Lie group, and it is also open in $W$ (cf. \cite[Lemma 2.1.4]{Pri77}).  It implies that $W$ has only finitely many connected branches. Let $\pi:S^1\times {\mathcal O}(d)\rightarrow S^1$ be the naturally coordinate projection. Since $\gamma$ is an irrational number,  $\pi(W)=S^{1}$, and hence $\pi(W_0)$ is a subgroup of $S^1$ with positive Haar measure (for $W$ has only finitely many connected branches).  It follows that $\pi(W_0)=S^1$.
Then there is one parameter subgroup $t\in \mathbb{R}\mapsto (e^{2\pi i t},\phi(t))\in W_0$, where $\phi:\mathbb{R}\rightarrow {\mathcal O}(d)$ is an analytic group homomorphism (cf. \cite[Theorems 2.2.10, 2.2.12]{Pri77}). Therefore $\phi$ is an analytic function.
By  \eqref{eq-11}, we have
\begin{equation}\label{eq-2}
U\supseteq \left\{\alpha^{2(p+\{t\})}|MOv|^2:(e^{2\pi i t},O)\in W\right\}\supseteq\left\{\alpha^{2(p+\{t\})}|M\phi(t)v|^2:t\in \mathbb{R}\right\}.
\end{equation}
Put $f(t)=\alpha^{2(p+t)}|M\phi(t)v|^2$, $t\in\mathbb{R}$. Then $f$ is a positive analytic function on $\mathbb{R}$ since $\phi$ is analytic.
Notice that $\lim_{t\rightarrow +\infty}f(t)=0$. Hence $f$ is not constant on any  non-degenerate interval of $\mathbb{R}$.
Then $J:=\{ f(t):t\in [0,\frac{1}{2})\}$ is a non-degenerate interval of $\mathbb{R}$. Clearly $U\supseteq J$. Thus
$\dim_H(E-E)\ge \dim_HU\ge \dim_HJ=1$.

Next we consider the case that $\beta_j$, $1\leq j\leq m$, might be different.  Without loss of generality, we show that $\log \beta_1/\log \alpha\in \Q$.
Since $F$ is not a singleton, there exists $j\geq 2$ such that the fixed point of  $\psi_j$ is different that of $\psi_1$. Let $F_1$ be the attractor of the IFS
$\{\psi_1\circ \psi_j, \psi_j\circ \psi_1\}$. Then $F_1\subset F$ is not a singleton and can be affinely embedded into $E$, hence $\log (\beta_1\beta_j)/\log \alpha\in \Q$. Similarly considering the IFS
$\{\psi_1^2\circ \psi_j, \psi_j\circ\psi_1^2\}$,  we have $\log (\beta_1^2\beta_j)/\log \alpha\in \Q$. Hence $\log \beta_1/\log \alpha\in \Q$.
This ends the proof of Theorem \ref{thm-1.3}.
\end{proof}

Applying Theorem \ref{thm-1.4} to the case that  $E, F$ are  central Cantor sets with $0<\beta<\alpha <\frac{1}{4}$, we see that  if $\C_\beta $  can be affinely embedded into $\C_\alpha$ then $\frac{\log \beta}{\log
\alpha}\in \mathbb{Q}$. Theorem \ref{thm-1.4}(i) sharpens this result.

\begin{proof}[Proof of Theorem \ref{thm-1.4}(i)] Let $0<\beta<\alpha <\frac{1}{4}$. If $\frac{\log \beta}{\log
\alpha}\in \mathbb{N}$, then it is clear that $\C_\beta\subseteq \C_\alpha$, hence $\C_\beta$ can be affinely embedded into $\C_\alpha$.

 Conversely, assume that $\C_\beta$ can be affinely embedded into  $\C_\alpha$.
 Assume that $\frac{\log \beta}{\log \alpha}\not \in \mathbb{N}$. We will derive a contradiction as below.

 Since $0<\alpha<1/4$, we have $\sqrt{\alpha}<1-2\alpha$. If
$\frac{\log \beta}{\log \alpha}\not \in \mathbb{N}$, then there
exists a prime number $\ell\in \mathbb{N}$ such that $\frac{\log
\gamma}{\log \alpha}\not \in \mathbb{N}$ and
$\sqrt{\alpha}<\frac{1-2\alpha}{1-2\gamma}$, where
$\gamma=\beta^\ell$. Since $\C_\beta$ contains a nontrivial affine
image of $\C_\gamma$, $\C_\alpha$ contains a nontrivial affine image
of $\C_\gamma$. That is, there exists $a\in [0,1]$ and $\lambda\neq 0$
such that $a+\lambda \C_\gamma\subseteq \C_\alpha$. We can assume that $\lambda>0$ (since  $\lambda+(-\lambda)\C_\gamma=\lambda\C_\gamma$).

Next, we are to show that there exist $m,n\in \mathbb{N}$ such that
\begin{equation}\label{eee-1}
\alpha< \lambda
\frac{\gamma^m}{\alpha^n}<\frac{1-2\alpha}{1-2\gamma}.
\end{equation}
Notice that $\frac{\gamma^m}{\alpha^n}=e^{ (m\log
\gamma-n\log \alpha)}=e^{\log \alpha (m\frac{\log \gamma}{\log
\alpha}-n)}$. When $\frac{\log \gamma}{\log \alpha}$ is an irrational
number, then $\left\{m\frac{\log \gamma}{\log \alpha}-n:m,n\in
\mathbb{N}\right\}$ is dense in $\mathbb{R}$. So there exist $m,n\in
\mathbb{N}$ such that $\alpha< \lambda
\frac{\gamma^m}{\alpha^n}<\frac{1-2\alpha}{1-2\gamma}$ since
$\alpha<\frac{1-2\alpha}{1-2\gamma}$. Next assume that $\frac{\log \gamma}{\log \alpha}\in \mathbb{Q}$. Since  $\frac{\log \gamma}{\log \alpha}\not \in \mathbb{N}$,
there exist two coprime integers $p> q\ge 2$
such that  $\gamma=\alpha^{{p}/{q}}$. Since $
\lambda\alpha^{{k}/{q}}\searrow 0$ when $k\nearrow +\infty$ and
$ \lambda\alpha^{{k}/{q}}\nearrow +\infty$ when $k\searrow
-\infty$, there exist $r\in \mathbb{Z}$ such that
$\lambda\alpha^{{r}/{q}}<\frac{1-2\alpha}{1-2\gamma}$ and
$\lambda\alpha^{{(r-1)}/{q}}\ge \frac{1-2\alpha}{1-2\gamma}$. Thus
$\lambda\alpha^{{r}/{q}}\ge
\frac{1-2\alpha}{1-2\gamma}\alpha^{{1}/{q}}\ge
\frac{1-2\alpha}{1-2\gamma}\alpha^{{1}/{2}}>\alpha$. That is,
$\alpha<\lambda\alpha^{{r}/{q}}<\frac{1-2\alpha}{1-2\gamma}$.
Hence we can find $m,n\in \mathbb{N}$ such that $$\alpha< \lambda
\frac{\gamma^m}{\alpha^n}=\lambda\alpha^{{r}/{q}}<\frac{1-2\alpha}{1-2\gamma},$$
since $\left\{ \frac{\gamma^m}{\alpha^n}:m,n\in \mathbb{N}\right\}=\left\{
\alpha^{{k}/{q}}: k\in \mathbb{Z}\right\}$. This proves \eqref{eee-1}.

 Notice that $\Phi=\{\phi_1=\alpha x,\; \phi_2=\alpha x+(1-\alpha)\}$ is a generating IFS of $\C_\alpha$.
Denote $$\A_k:=\{\phi_I(\C_\alpha):\; I\in \{1,2\}^k\}$$ for  $k\in \N$. Clearly, $\A_k$ is a cover of $\C_\alpha$, and any two different sets in $\A_k$ have a distance $\geq \alpha^{k-1}(1-2\alpha)$.

Let $(m,n)$ be a pair in $\N^2$ so that \eqref{eee-1} holds. Then
$$
H:=a+\lambda \gamma^m\C_\gamma\subset a+\lambda\C_\gamma\subseteq \C_\alpha=\bigcup_{A\in \A_{n+1}}A.
$$
Since ${\rm diam}(H)=\lambda\gamma^m> \alpha^{n+1}$ by \eqref{eee-1},  $H$ intersects at least two elements in $\A_{n+1}$. Therefore $H$ contains a ``hole'' of length $\geq \alpha^n(1-2\alpha)$.  However by the geometric structure of $\C_\gamma$, the longest ``hole'' in $H$ is of length $\lambda \gamma^m(1-2\gamma)$, which is less than
 $\alpha^n(1-2\alpha)$ by \eqref{eee-1}. Hence we derive a contradiction.  This finishes the proof of Theorem \ref{thm-1.4}(i).
\end{proof}

\section{Unique beta expansions and affine embeddings}
\label{S-4}
In this section, we first  show that
 there exist $0<\beta<\alpha<\frac{1}{2}$ with $\frac{\log
\beta}{\log \alpha}\in \Q\backslash \mathbb{N}$ such that $\C_\beta$ can be affinely embedded into $\C_\alpha$. Then we derive some unusual behavior if $\C_\beta$ can be affinely embedded into $\C_\alpha$ and $\frac{\log
\beta}{\log \alpha}\not \in \mathbb{Q}$. In the end, we combine these behaviors and certain uniqueness property of  beta-expansions to prove Theorem \ref{thm-1.4}(ii).

\begin{lem}\label{ex-1}
For  $k\ge 2$, let $\alpha_k>0$ be the unique positive
solution of the equation
$$\sqrt{x}=x+x^2+\cdots+x^k.$$
Set $\beta_k=\alpha_k^{{(2k+1)}/{2}}$.
Then $\frac{1}{2}>\alpha_2>\alpha_3>\cdots$,  $\lim
\limits_{k\rightarrow +\infty} \alpha_k=\frac{3-\sqrt{5}}{2}$, and $\C_{\beta_k}$ can be affinely embedded into $\C_{\alpha_k}$.
\end{lem}
\begin{proof}
Fix $k\ge 2$.  Denote $\lambda_k=\frac{1-\alpha_k}{1-\beta_k}$. For
any $z=\sum_{i=0}^{+\infty} \epsilon_i(1-\beta_k)\beta_k^i\in \C_{\beta_k}$ with $\epsilon_i\in \{0,1\}$ for $i\geq 0$, we have
\begin{align*}
\lambda_k z &=\sum_{j=0}^\infty (1-\alpha_k)\Big(\epsilon_{2j}\beta_k^{2j}+\epsilon_{2j+1}\beta_k^{2j+1}\Big)\\
&=\sum_{j=0}^\infty
(1-\alpha_k)\Big(\epsilon_{2j}\alpha_k^{(2k+1)j}+\epsilon_{2j+1}\alpha_k^{(2k+1)j+k}\sqrt{\alpha_k}\Big)\\
&=\sum_{j=0}^\infty
(1-\alpha_k)\Big(\epsilon_{2j}\alpha_k^{(2k+1)j}+\epsilon_{2j+1}\alpha_k^{(2k+1)j+k}(\alpha_k+\cdots+\alpha_k^k)\Big)
\\
&=\sum_{j=0}^\infty
(1-\alpha_k)\Big(\epsilon_{2j}\alpha_k^{(2k+1)j}+\epsilon_{2j+1}\alpha_k^{(2k+1)j+k+1}+\cdots+\epsilon_{2j+1}\alpha_k^{(2k+1)j+2k}\Big).
\end{align*}
Hence $\lambda_kz\in \C_{\alpha_k}$. Thus $\lambda_k
\C_{\beta_k}\subset \C_{\alpha_k}$.
\end{proof}

\bigskip

Let $0<\beta<\alpha<1/2$. Let $(\Sigma,\sigma)$ denote the full shift over the alphabet $\{0,1\}$. That is, $\Sigma=\{0,1\}^\mathbb{N}$ and $\sigma$ is the left shift on $\Sigma$.
Let $\pi:\Sigma\rightarrow [0,1]$ be the coding
map defined as
$$\pi(z)=(1-\alpha)\sum_{i=0}^\infty z_i \alpha^i,\qquad z=(z_i)_{i=0}^\infty\in \Sigma.$$
 Clearly, $\pi$ is one-to-one and
$\pi(\Sigma)=\C_\alpha$. Notice that $\C_\alpha$ has a generating IFS $\{S_0, S_1\}$,  where $S_0(x)=\alpha x$, $S_1(x)=\alpha x +1-\alpha$. It is direct to check that
$$\pi(z)=\lim_{n\to \infty}S_{z_0}\circ\cdots \circ S_{z_{n-1}}(0)$$ for $z\in \Sigma$.

\begin{lem}\label{lem-4.1} Assume that $\frac{\log\beta}{\log \alpha}\not \in
\mathbb{Q}$ and there exist $a\ge 0, \;\lambda>0$ such that
\begin{align}\label{eq-1}
a+\lambda \C_\beta\subseteq \C_\alpha.
\end{align}
Let $z= \pi^{-1}(a)$. Then the following properties hold.
\begin{itemize}
\item[(i)] For any $n,k\in \mathbb{N}$, $\pi(\sigma^nz)+\lambda\frac{\beta^k}{\alpha^n}\C_\beta\subseteq \C_\alpha$
provided that $\frac{\lambda \beta^k}{\alpha^n}<\frac{1-2\alpha}{\alpha}$.
\item [(ii)] For any $u\in
[0,\frac{1-2\alpha}{\alpha}]$, there exists $w\in
\overline{\{\sigma^nz:n\in \mathbb{N}\}}$ such that
$$\pi(w)+u\C_\beta\subseteq \C_\alpha.$$
\item[(iii)]$\dim_H(\pi(\overline{\{\sigma^nz:n\in \mathbb{N}\}}))\ge 1-\dim_H \C_\alpha>0.$
\end{itemize}
\end{lem}
\begin{proof} Let $k, n\in \N$ satisfy $\lambda\beta^k<(1-2\alpha)
\alpha^{n-1}$.   Notice that  $a\in S_{z_0}\circ \cdots \circ
S_{z_{n-1}}(\C_\alpha)$ and  $${\rm dist}\big(\C_\alpha\setminus
S_{z_0}\circ \cdots \circ S_{z_{n-1}}(\C_\alpha),\; S_{z_0}\circ \cdots
\circ S_{z_{n-1}}(\C_\alpha)\big)\ge (1-2\alpha)\alpha^{n-1}.$$
Since $a+\lambda \beta^k \C_\beta\subseteq a+\lambda
\C_\beta\subseteq \C_\alpha$, and the diameter of  $a+\lambda \beta^k \C_\beta$ is less than $(1-2\alpha)
\alpha^{n-1}$, we have
$$a+\lambda\beta^k\C_\beta\subseteq S_{z_0}\circ \cdots S_{z_{n-1}}(\C_\alpha).$$
 Hence,  $(S_{z_0}\circ \cdots
S_{z_{n-1}})^{-1}(a)+\frac{\lambda\beta^k}{\alpha^n}\C_\beta\subseteq
\C_\alpha$, i.e., $\pi(\sigma^n z)+\frac{\lambda\beta^k}{\alpha^n}\C_\beta\subseteq
\C_\alpha$. This proves (i).

 To prove (ii), let $u\in
[0,\frac{1-2\alpha}{\alpha}]$. Since $\frac{\log\beta}{\log
\alpha}\not \in \mathbb{Q}$, there exist pairs $(k_i,n_i)\in
\mathbb{N}^2$ such that $\frac{\lambda
\beta^{k_i}}{\alpha^{n_i}}<\frac{1-2\alpha}{\alpha}$ and
$\frac{\lambda \beta^{k_i}}{\alpha^{n_i}}\rightarrow
u$ as $i\rightarrow +\infty$. By (i),
$$\pi(\sigma^{n_i}z)+\lambda\frac{\beta^{k_i}}{\alpha^{n_i}}\C_\beta\subseteq \C_\alpha.$$
 Let $w$ be an accumulation  point of $(\sigma^{n_i}z)$. Then we have
$\pi(w)+u\C_\beta\subseteq \C_\alpha$. This proves (ii).

By (ii), for any $u\in
[0,\frac{1-2\alpha}{\alpha}]$,
$$
u\C_\beta\subset \C_\alpha-\pi(\overline{\{\sigma^nz:n\in \mathbb{N}\}}).
$$
Hence
$$\left[0,\frac{1-2\alpha}{\alpha}\right]\subseteq \C_\alpha-\pi(\overline{\{\sigma^nz:n\in \mathbb{N}\}}).$$ It follows that
$\dim_H\Big(\C_\alpha-\pi \big( \overline{\{\sigma^nz:n\in \mathbb{N}\}}\big) \Big)\ge
1$, thus $$\dim_H \C_\alpha+\dim_H\pi \big(\overline{\{\sigma^nz:n\in \mathbb{N}\}}\big) \ge 1.$$
This finishes the proof of (iii).
\end{proof}

\begin{prop}
\label{pro-4.1}
Let $0<\alpha<\sqrt{2}-1$ and $\displaystyle u^*=\frac{1-\alpha}{1+\alpha}$. Set $z_i=(-1)^i$ for $i\geq 0$. Then we have the following statements.
\begin{itemize}
\item[(i)]   $z:=(z_i)_{i=0}^\infty$ is the unique element in $\{0, 1,-1\}^\N$ so that
\begin{equation}
\label{e-4.11}(1-\alpha)\sum_{i=0}^\infty z_i\alpha^i= u^*.
\end{equation}
\item[(ii)] Let $a=(1-\alpha)\sum_{i=0}^\infty \alpha^{2i+1}$ and $b=(1-\alpha)\sum_{i=0}^\infty \alpha^{2i}$. Then $(a, b)$ is the unique point in $\C_\alpha\times \C_\alpha$ so that $u^*=b-a$.
\end{itemize}
\end{prop}
\begin{proof}
It is direct to verify \eqref{e-4.11}. Assume $w=(w_i)_{i=0}^\infty$ is a point in $\{0, 1,-1\}^\N$ so that  $(1-\alpha)\sum_{i=0}^\infty w_i\alpha^i= u^*$ (i.e.
$\sum_{i=0}^\infty w_i\alpha^i=1/(1+\alpha)$). We show below that $w=z$.

Since $0<\alpha<\sqrt{2}-1$, we have $2\alpha+\alpha^2<1$, and thus  $$\sum_{i=1}^\infty w_i\alpha^i\leq \sum_{i=1}^\infty \alpha^i=\frac{\alpha}{1-\alpha}<\frac{1}{1+\alpha}.$$
It follows that $w_0=1$. Similarly,
$$
w_0+\sum_{i=2}^\infty w_i\alpha^i\geq 1-\sum_{i=2}^\infty \alpha^i=1-\frac{\alpha^2}{1-\alpha}>\frac{1}{1+\alpha}.$$
It follows that $w_1=-1$.
Now we have $$\sum_{i=2}^\infty w_i\alpha^{i}=1/(1+\alpha)-(w_0+w_1\alpha)=\alpha^2/(1+\alpha).$$ Hence $\sum_{i=2}^\infty w_i\alpha^{i-2}={1}/{(1+\alpha)}.$
It follows that $w_2=1$ and $w_3=-1$. Repeating the above argument, we have $w=z$. This finishes the proof of (i).

To show (ii), we first notice that $u^*=b-a$. Now assume that $u^*=b'-a'$ for some pair $(b', a')\in \C_\alpha\times \C_\alpha$. Then there exist $e=(e_i)_{i=0}^\infty$ and $f=(f_i)_{i=0}^\infty\in \{0,1\}^\N$ such that $b'=\pi(f)$ and $a'=\pi(e)$. Hence
$$u^*=b'-a'=(1-\alpha)\sum_{i=0}^\infty(f_i-e_i)\alpha^i.$$
By (i), we have $f_i-e_i=(-1)^i$ for $i\geq 0$. This forces that $$e_i=\begin{cases} 1 &\text{ if $i$ is odd}\\
0 &\text{ if $i$ is even} \end{cases}\quad  \text{ and }\quad  f_i=\begin{cases} 0 &\text{ if $i$ is odd}\\
1 &\text{ if $i$ is even} \end{cases}.$$
Hence $b'=b$ and $a'=a$. This proves (ii) and we are done.
\end{proof}

\begin{proof}[Proof of Theorem \ref{thm-1.4}(ii)] Assume that $1/4<\alpha<\sqrt{2}-1$ and $0<\beta<\alpha$. Assume that $\C_\beta$ can be affinely embedded into $\C_\alpha$. By Lemma \ref{ex-1}, we see that it is possible that $\log \beta/\log\alpha\not \in \N$. In the following we show that $\log \beta/\log \alpha\in \Q$.

Assume on the contrary that $\log \beta/\log \alpha\not \in \Q$. Define $u^*$ as in Proposition \ref{pro-4.1}. Then $u^*\in (0,\frac{1-2\alpha}{\alpha})$.  By Lemma \ref{lem-4.1}(ii), there exists $a\in \C_\alpha$ such that \begin{equation}
\label{e-4.12}
a+u^*\C_\beta\subseteq \C_\alpha.
 \end{equation}
 In particular, $b:=a+u^*\in \C_\alpha$. Hence $u^*=b-a$ with $a, b\in \C_\alpha$.   By Proposition \ref{pro-4.1}, we must have $a=(1-\alpha)\sum_{i=0}^\infty \alpha^{2i+1}$. Let $z=\pi^{-1}(a)$. Then $z=(01)^\infty$ is a periodic point in $\Sigma=\{0,1\}^\infty$. Hence $\dim_H(\pi(\overline{\{\sigma^nz:n\in \mathbb{N}\}}))=0$. However, according to \eqref{e-4.12} and Lemma \ref{lem-4.1}(iii), we must have $\dim_H(\pi(\overline{\{\sigma^nz:n\in \mathbb{N}\}}))>0$. This leads to a contradiction.
\end{proof}

\section{Sets of uniqueness and affine embeddings in the Pisot case}
\label{S-5}
In this section, we prove Theorem \ref{thm-1.4}(iii) and Theorem \ref{thm-1.5}.
Our proofs make use of the theory of  sets of uniqueness for trigonometric series. In the following we give some necessary definitions and theorems (see, e.g., \cite{Sal63, KeLo87} for details).

\begin{de}
\label{de-5.1} A set $E\subseteq [0,2\pi]$ is called
a {\it set of uniqueness} if every trigonometric series $\sum_{n=0}^{+\infty}
\big( a_n\cos (nx)+b_n \sin (nx)\big)$ which converges to zero on  $[0,2\pi]\backslash E$ is identically $0$, i.e., $a_n=b_n=0$ for all $n\geq 0$. Otherwise $E$ is called  a set of multiplicity.
\end{de}
\begin{rem} It is clear that any subset of a set of uniqueness is
still a set of uniqueness.
\end{rem}

One  fundamental problem in classical harmonic analysis is to characterize the  sets of uniqueness. So far this problem is still open in its full generality. A major achievement  was made by Salem and Zygmund in 1955 to characterize when a homogeneous Cantor set is a set of uniqueness.

\begin{thm}[Salem and Zygmund, cf. Chap. VII of \cite{Sal63}]\label{thm-5.1} Let $0<\alpha<1/2$.   Suppose that $E\subset [0,2\pi]$ is  the attractor of an IFS
$\{\alpha x+a_i\}_{i=1}^\ell$, where  $2\leq \ell<1/\alpha$ and  $0=a_1<a_2<\ldots<a_\ell=1-\alpha$. \footnote{An additional assumption that $a_{i+1}-a_i>\alpha$ was put by Salem and Zygmund. Nevertheless, their proof did not use this assumption.} Then $E$ is a set of uniqueness if and only if
\begin{itemize}
\item[(i)] $\theta=1/\alpha$ is  a Pisot number,
\end{itemize}
and
\begin{itemize}
\item[(ii)]  $a_1, \ldots, a_\ell$ are  algebraic numbers in the field of $\theta$ over $\Q$.
\end{itemize}
\end{thm}

Besides using the above theorem, we are going to use the following properties of the sets of uniqueness.

\begin{thm}[cf.  pp. 2, 71 of \cite{KeLo87}]
\label{thm-5.2}
\begin{itemize}
\item[(i)] The union of countably many closed sets of uniqueness is also a set of uniqueness.
\item[(ii)]
The sets of uniqueness are closed under translations, dilations and contractions.  That is, if $E, F\subset[0,2\pi]$ and $E=\lambda F+ a$ is an affine copy of $F$, then
$E$ is a set of uniqueness if and only if $F$ is a set of uniqueness.
\end{itemize}
\end{thm}

\begin{thm}[Menshov, cf. p. 46 of \cite{Sal63}] \label{thm-5.3} If $E\subset [0,2\pi]$ is a closed set of
uniqueness, then $\widehat{\eta}(n)\not \rightarrow 0$ as
$|n|\rightarrow +\infty$ for any Borel probability measure $\eta$
supported on $E$, where $\widehat{\eta}(n)=\int e^{-inx}d\mu(x)$ are the Fourier coefficients of
$\eta$.
\end{thm}


\bigskip

Now we are ready to prove Theorem \ref{thm-1.4}(iii).
\begin{proof}[Proof of Theorem \ref{thm-1.4}(iii)] Let $0<\beta<\alpha<1/2$ and  $1/\alpha$ be a Pisot number.  By Theorem \ref{thm-5.1}, $\C_\alpha$ is a set of uniqueness.  Now assume that $\C_\beta$ can be affinely embedded into $\C_\alpha$. By Theorem \ref{thm-5.2}(ii), $\C_\beta$ is also a set of uniqueness. Thus by Theorem \ref{thm-5.1}, $1/\beta$ is a Pisot number.

Next we  prove that
$\frac{\log\beta}{\log \alpha}\in \mathbb{Q}$. Assume on the contrary that $\frac{\log\beta}{\log \alpha}\not\in \mathbb{Q}$. We will derive a contradiction as follows.

Denote $b=\frac{1-2\alpha}{\alpha}$. By Lemma \ref{lem-4.1}(ii),  for any $u\in [0,b]$, there exists $c\in \mathbb{R}$ such that
\begin{align}\label{affine-c}
u\C_\beta+c\subseteq \C_\alpha.
\end{align}
Define $f: [0, b]\to \C_\alpha$ by $f(u)=\sup\{d\in \R:\;u\C_\beta+d\subseteq \C_\alpha\}$. A compactness argument shows that  $f$ is upper semi-continuous.

Let $\mu$ denote the normalized $\frac{\log 2}{\log (1/\beta)}$-dimensional Hausdorff measure restricted on $\C_\beta$.
For $u\in [0,b]$, define $h_u:\mathbb{R}\rightarrow \mathbb{R}$ by $h_u(x)=ux+f(u)$.
Let $\eta=\frac{1}{b}\int_0^b \mu\circ h_u^{-1} d u$, i.e., $$\eta(A)=\frac{1}{b}\int_0^b \mu\circ h_u^{-1}(A) du$$ for any Borel set $A\subseteq \mathbb{R}$.
Since $u\C_\beta+f(u)\subseteq \C_\alpha$ for each $u\in [0,b]$, $\eta$ is supported on $\C_\alpha$.

Now let us estimate the Fourier coefficients of $\eta$. For $n\in {\Bbb Z}$,
\begin{align*}
\widehat{\eta}(n)&=\int_{\mathbb{R}}e^{-in x} d\eta(x)
=\frac{1}{b}\int_0^b  \int_{\mathbb{R}}e^{-in x} d\mu\circ h_u^{-1}(x)du\\
&=\frac{1}{b}\int_0^b \int_{\mathbb{R}}e^{-in (ux+f(u))} d\mu(x)  du\\
&=\frac{1}{b}\int_0^b \widehat{\mu}(un)e^{-if(u)n} du,
\end{align*}
where $\widehat{\mu}(\xi):=\int e^{-i\xi x} d\mu(x)$ for $\xi\in \R$.
Hence
\begin{align}
\label{e-5.1}
|\widehat{\eta}(n)|\le \frac{1}{b}\int_0^b |\widehat{\mu}(un)|du\le \frac{1}{b|n|}\int_{-b|n|}^{b|n|} |\widehat{\mu}(x)|dx
\end{align}
for $n\neq 0$. However, since $\mu$ does not contain  atomics, by Wiener Theorem \cite{Wie59},
\begin{align*}
\limsup_{T\rightarrow +\infty}\frac{1}{T}\int_{-T}^T |\widehat{\mu}(x)|^2dx=0.
\end{align*}
Applying the Cauchy-Schwartz inequality, we have  $\limsup_{T\rightarrow +\infty}\frac{1}{T}\int_{-T}^T |\widehat{\mu}(x)|dx=0$.
Thus by \eqref{e-5.1}, $\widehat{\eta}(n)\to 0$ as $|n|\to \infty$. However, this contradicts Theorem \ref{thm-5.3}, since $\C_\alpha$ is a set of uniqueness and $\eta$ is supported on $\C_\alpha$.
\end{proof}

\begin{rem}
\label{rem-5.1} As an extension of Theorem \ref{thm-1.4}(iii), the following statement can be proved by using the same argument:
 Suppose that $B\subset\R$ is a compact set of uniqueness and $A\subset \R$  a compact set which supports a continuous Borel probability measure. Then for any $\epsilon>0$, there exists $\delta\in (0,\epsilon)$ such that $B$ does not  contain any translation of $\delta A$.

\end{rem}

\begin{lem}
\label{lem-5.6}
Let $\theta>1$ be a Pisot number and $\A$ be a finite subset of $\Z[\theta]$. Then there exists a constant $C>0$ such that for any $n\in \N$ and $t_1,\ldots, t_n\in \A$,
$$
\mbox{either } \quad \sum_{i=1}^n t_i\theta^i=0 \quad \mbox{or }\quad  \left| \sum_{i=1}^n t_i\theta^i\right|\geq C.
$$
\end{lem}
\begin{proof}
The result was essential due to Garsia \cite[Lemma 1.52]{Gar62}. For completeness, we provide a proof.

 We denote by $\theta^{(1)},\ldots, \theta^{(k)}$  the algebraic conjugates of $\theta$, and by $t^{(1)}, \ldots, t^{(k)}$  the conjugates of $t\in \A$. Since $\theta$ is a Pisot number,  $$\rho:=\max_{1\leq j\leq k}|\theta^{(j)}|<1.$$
Corresponding, for each $t\in \A$, we denote by $t^{(1)}, \ldots, t^{(k)}$  the conjugates of $t$.

Let $t_1,\ldots, t_n\in \A$. Assume that $\sum_{i=1}^n t_i\theta^i\neq 0$. Then
$$
\left(\sum_{i=1}^n t_i\theta^i\right)\left(\prod_{j=1}^k \sum_{i=1}^n  t_i^{(j)}(\theta^{(j)})^i\right)
$$
is a non-zero integer. Hence
$$
\left|\sum_{i=1}^n t_i\theta^i\right|\geq \frac{1}{ \prod_{j=1}^k \sum_{i=1}^n  |t_i^{(j)}(\theta^{(j)})^i|}\geq \frac{(1-\rho)^k}{\max\{|t^{(j)}|^k:\; 1\leq j\leq k, \; t\in \A\}}=: C.
$$
This finishes the proof.
\end{proof}

\begin{proof}[Proof of Theorem \ref{thm-1.5}] Denote $\D=\{a_i:\; i=1,\ldots, \ell\}$ and
$$
\Lambda:=\left\{\sum_{i=1}^n t_i\theta^i:\; n\in \N,\;  t_1,\ldots, t_n\in \D-\D\right\}.
$$
Then  $\Lambda-\Lambda=\left\{\sum_{i=1}^n t_i\theta^i:\; n\in \N,\;  t_1,\ldots, t_n\in \A\right\}$, with $\A:=(\D-\D)-(\D-\D)$.
By Lemma \ref{lem-5.6}, $\Lambda-\Lambda$ is uniformly discrete. Hence  $\Lambda\cap [0, 1]$ is a finite set.

In the following we first prove Theorem \ref{thm-1.5} in the case that $F$ is homogeneous in the sense that  $\beta_jO_j(x)=\beta x$ for some $\beta>0$ and all $1\leq j\leq m$.  Since  $F$ can be affinely embedded into $E$,  there exist $a, \lambda\in \R$ such that $\lambda\neq 0$ and $a+\lambda F\subseteq E$. Without loss of generality, we may assume that $\lambda>0$ (notice that  $-F$ is also a homogeneous self-similar set). By Theorem \ref{thm-5.1}, $E$ is a set of uniqueness. Hence by Theorem \ref{thm-5.2}, $F$ is also a set of uniqueness.  Applying Theorem \ref{thm-5.1} to $F$, we see that $1/\beta$ is a Pisot number.  Assume that $\log \beta/\log \alpha\not\in \Q$. We derive a contradiction as below.

We claim that for any $u\in (0, 1/{\rm diam(F)}]$, there exists $c=c_u\in \R$ such that
\begin{equation}
\label{e-5.7}
uF+c\subseteq E+\F,
\end{equation}
where $\F:=\Lambda\cap [0, 1]$. To see this,  let $u\in (0, 1/{\rm diam(F)}]$.  Since $\log \beta/\log \alpha\not\in \Q$, there exist pairs $(k_i,n_i)\in
\N^2$ such that $\frac{\lambda
\beta^{k_i}}{\alpha^{n_i}}<1/{\rm diam(F)}$ and
$\frac{\lambda \beta^{k_i}}{\alpha^{n_i}}\rightarrow
u$ as $i\rightarrow +\infty$. However, for each $i\in \N$,
we have $$a+\lambda \psi_{1^{k_i}}(F)\subset a+\lambda F \subseteq E=\bigcup_{I\in \{1,\ldots, \ell\}^{n_i}}\phi_I(E).$$
Notice that $a+\lambda \psi_{1^{k_i}}(F)=\lambda \beta^{k_i}F+d_i$ for some $d_i\in \R$, and
$$
\bigcup_{I\in \{1,\ldots, \ell\}^{n_i}}\phi_I(E)=\alpha^{n_i}\left(E+\D_{n_i}\right),
$$
with $\D_{n_i}:=\left\{\sum_{j=1}^{n_i} t_j\theta^j:\; t_1,\ldots, t_{n_i}\in \D\right\}$. Let $c_{i}$ be the smallest element in $\D_{n_i}$ so that
$(\lambda \beta^{k_i}F+d_i) \cap  \alpha^{n_i}(E+c_i)\neq \emptyset$. As the diameter of $\lambda \beta^{k_i}F+d_i$ is less than $\alpha^{n_i}$,  we have
$(\lambda \beta^{k_i}F+d_i) \cap  \alpha^{n_i}(E+t)=\emptyset$ for any $t\in \D_{n_i}$ satisfying $t<c_i$ or $t>c_i+1$.  Thus,
$$
\lambda \beta^{k_i}F+d_i\subseteq \alpha^{n_i}\left(E+c_i+(\D_{n_i}-c_i)\cap [0,1]\right)\subseteq \alpha^{n_i}\left(E+c_i+\F\right).
$$
Therefore $\lambda \beta^{k_i}\alpha^{-n_i}F+e_i\in E+\F$ for some $e_i\in \R$. Letting $i\to \infty$, we have $uF+c \subseteq E+\F$, where $c$ is an accumulation point of $(e_i)$. This proves the claim.

 Since $E$ is a set of uniqueness and $\F$ is a finite set, by Theorem \ref{thm-5.2}, $E+\F$ is also a set of uniqueness.   By \eqref{e-5.7} and Remark \ref{rem-5.1}, we get a contradiction. Therefore, we have proved Theorem \ref{thm-1.5} in the case that $F$ is homogeneous.

Next we consider the case that $\beta_j$, $1\leq j\leq m$, might be different.  Without loss of generality, we show that $\log \beta_1/\log \alpha\in \Q$ and $1/\beta_1$ is  a Pisot number. We first repeat some argument used in the last paragraph of the proof of Theorem \ref{thm-1.3}.
Since $F$ is not a singleton, there exists $j\geq 2$ such that the fixed point of  $\psi_j$ is different that of $\psi_1$. Let $F_1$ be the attractor of the IFS
$\{\psi_1\circ \psi_j, \psi_j\circ \psi_1\}$. Then $F_1\subset F$ is not a singleton and can be affinely embedded into $E$, hence $\log (\beta_1\beta_j)/\log \alpha\in \Q$. Similarly considering the IFS
$\{\psi_1^2\circ \psi_j, \psi_j\circ\psi_1^2\}$,  we have $\log (\beta_1^2\beta_j)/\log \alpha\in \Q$. Hence $\log \beta_1/\log \alpha\in \Q$.

  To see that $1/\beta_1$ is  a Pisot number, we notice that for any $n,m\in \N$, $\beta_1^{-n}\beta_j^{-m}$ is a Pisot number (since the attractor of the homogeneous IFS $\{\psi_1^n\circ \psi_j^m,\; \psi_j^m\circ \psi_1^n\}$ can be affinely embedded into $E$). We also notice that $\log \beta_j/\log \beta_1\in \Q$ (since $\log \beta_j/\log \alpha,\ \log\beta_1/\log \alpha\in \Q$). Write $\beta_j=\beta_1^{u/v}$, where $u,v$ are co-prime positive integers.  Then  for any $n\in \N$, $\beta_1^{-n-u}=\beta_1^{-n}\beta_j^{-v}$ is a Pisot number.
Let $f(x)$ be the minimal integer polynomial for $\xi:=1/\beta_1$. Let $\xi_1,\ldots, \xi_k$ denote the algebraic conjugates of $\xi$, and set $\xi_0:=\xi$.
Take an integer $p>u$ so that $$e^{2\pi i/p}\not\in \{\xi_i/\xi_j:\; 0\leq i,j\leq k\}.$$ Then $\xi_i^p$, $i=0,\ldots, k$, are distinct.
As the Galois group of the minimal polynomial $f$  for $\xi$ is transitive, so for any $1\leq i\leq k$, there is
an automorphism $h$ of the Galois group mapping $\xi$ to $\xi_i$. Let $g$ be the minimal integer polynomial for $\xi^p$. Then   $g(\xi_i^p)=g(h(\xi)^p)=h(g(\xi^p))=0$.
Hence  $\xi^p_i$ ($i=1,\ldots, k$) are algebraic conjugates of $\xi^p$ . Since $\xi^p$ is a Pisot number, we have $|\xi^p_i|<1$ for $1\leq i\leq k$. Hence $|\xi_i|<1$ for $1\leq i\leq k$.
It follows that $\xi=1/\beta_1$ is a Pisot number.
\end{proof}

\noindent{\bf Acknowledgements}. The authors would like to thank Julien Barral for  providing an  abstract in French. They  are grateful to  Hochman and Shmerkin for bringing up to their attention  Furstenberg's conjectures on intersections of Cantor sets \cite{Fur69}. The first author was partially supported by the RGC grant  in CUHK. The second author was partially supported by NNSF for
Distinguished Young Scholar (11225105), Fok Ying Tung Education Foundation and the Fundamental Research
Funds for the Central Universities (WK0010000014). The third author was
partially supported by  NNSF (10501002
and 11171128).

\end{document}